\documentclass[a4paper,10pt]{amsart}

\usepackage[utf8]{inputenc}
\usepackage[english]{babel}
\usepackage[T1]{fontenc}
\usepackage{amsfonts}
\usepackage{amsmath}
\usepackage{amsthm}
\usepackage{amssymb}
\usepackage{hyperref}
\usepackage{syntonly}
\usepackage{mathtools}
\usepackage{algorithm}
\usepackage{algpseudocode}
\usepackage{tikz}
\usepackage{dsfont}
\usepackage{braket}
\usepackage{faktor}
\usepackage{array}
\usepackage{lscape}
\usepackage{rotating}
\usepackage{epigraph}
\usepackage{subfig}
\usepackage{indentfirst}
\usepackage{booktabs}
\usepackage{multirow}
\usepackage{caption}
\usepackage{tabu}
\usepackage{tabularx}
\usepackage{xcolor}
\usepackage{pgf}
\usepackage{tikz}
\usepackage{caption}
\usepackage{color}
\usepackage{enumitem}
\usepackage{mathtools}
\usepackage{float}
\usepackage{graphicx}
\usepackage[displaymath, tightpage]{preview}
\usepackage[toc,page]{appendix}
\usetikzlibrary{matrix}
\usetikzlibrary{arrows}

\usepackage[colorinlistoftodos, textwidth=4cm, shadow,disable]{todonotes}

\theoremstyle{plain}
\newtheorem{theorem}{Theorem}[section]
\newtheorem{lemma}{Lemma}[section]

\theoremstyle{definition}
\newtheorem{definition}{Definition}[section]

\theoremstyle{remark}

\newtheorem{note}{Note}[section]
\newtheorem{example}{Example}[section]

\title[A new presentation for the inner Tutte group of a matroid]
      {A new presentation for the inner Tutte group of a matroid}

\author[Elia Saini]{Elia Saini}
\address{(Elia Saini) Department of Mathematics, University of Fribourg, Chemin du Mus\'ee 23, CH-1700, Fribourg, CH.}
\email{elia.saini@unifr.ch}

\begin{document}

\begin{abstract}
The inner Tutte group of a matroid is a finitely generated abelian group 
introduced as an algebraic counterpart of Tutte's homotopy 
theory of matroids. The aim of this work is to provide a new presentation for this group with a 
set of generators that is smaller than those previously known.

\end{abstract}

\maketitle

 \section*{Introduction}
 
  Tutte groups of matroids provide an algebraic setting for the study 
of problems in matroid theory. They were first introduced and studied by Dress and Wenzel in
in \cite{DW89} and subsequent papers \cite{Wen89} and \cite{DW90}. 
In particular, in \cite{DW91} and \cite{Wen91} Dress and Wenzel
discussed how \emph{matroids with coefficients} arise from the construction of Tutte groups.

This algebraic approach has been successfully exploited 
in describing reorientation spaces of different types of ``decorated'' matroids. 
For instance, an oriented matroid can be seen as a decoration of a ``underlying'' matroid by means of the set 
of ``real signs'' $\{0,+,-\}$ (see \cite[Chapter 3]{Zie}). 
The space of reorientation classes of oriented matroids with given underlying matroid is 
described in \cite{GRS95} as an algebraically defined subset of homomorphisms from the inner Tutte group to 
the multiplicative group $\{+,-\}$. 
With similar techniques, Delucchi and the author gave in \cite{DS15} an analogous characterization for
spaces of phased matroids with given underlying matroid in terms of homomorphisms from the inner Tutte group to 
the multiplicative group $S^{1}$. 
In fact, making a parallel with oriented matroid theory, a phased matroid can be seen as a decoration of a 
``underlying'' matroid by means of elements of the set of ``phases'' $\{0\}\cup S^{1}$ 
(see \cite[Section 2]{AD12} and \cite[Section 1]{DS15}).
Moreover, the recent work of Baker \cite{Bak16} suggests that this framework can be 
fruitfully exploited to study reorientation spaces of matroids over hyperfields (reference work in progress).

The \emph{inner Tutte group} of a matroid $M$ was first defined in \cite{DW89} as the kernel of a 
natural homomorphism from the \emph{Tutte group} of $M$ to a free abelian group 
(compare Definition \ref{innerdef}).
Explicit formulas for the rank of these groups were obtained in 
\cite[Theorem 3.3]{BL10} and \cite[Theorem 8.1]{DS15}. This results, together with \cite[Theorem 5.4]{Wen89}, imply a full 
description of the inner Tutte group of all matroids with ground set of up to $7$ elements.

In \cite[Theorem 3]{GRS95} and \cite[Theorem 4]{GRS95} two presentations of the inner Tutte group are described. 
However, both these presentations have very large sets of generators. 

In this work we provide 
a presentation of the inner Tutte group of a matroid with a smaller set of 
generators (Definition \ref{TMI0}).  
Together with \cite[Theorem 8.1]{DS15}, 
this affords a relatively efficient computation of the rank of the inner Tutte group of a matroid.
The central idea of our proof (Theorem \ref{main}) is to get rid of redundancies exploiting a geometric 
interpretation of the generators of the inner Tutte group in terms of ``cross-ratios'' 
(see \cite[Corollary 2]{GRS95} and \cite[Theorem 6.4]{DS15} for furhter details). 

As an illustration of the size of our improvements, in Table \ref{table1}
we compare the number of generators of the presentation 
of the inner Tutte group given in \cite[Theorem 4]{GRS95} 
with the number of generators of our presentation for some remarkable matroids.  

    \smallskip
   \noindent \textbf{Overview.} Section \ref{Section1} provides some basic definitions on matroids and Tutte groups. 
In Section \ref{Section2}, to each matroid we associate our ``smaller'' abelian group, 
presented by generators and relations, and we prove that it is isomorphic to the inner Tutte group of the given matroid.
For readibility's sake some technical computations are postponed to Appendix \ref{AppendixA}.

   \smallskip
  \noindent \textbf{Acknowledgments.} The author would like to express his gratitude to his
advisor Emanuele Delucchi and to Linard Hoessly and Ivan Martino for the very helpful discussions throughout 
all the time this project has been developed.  
This work is supported by Swiss National Science Foundation grant PP00P2\_150552/1.

 \section{Matroids and arrangements}\label{Section1}
   In this section we quickly provide some basic notions 
about matroids and Tutte groups. We point to the book \cite{Oxl92} for a 
detailed treatment of matroid theory and we refer to \cite{DW89} and subsequent 
papers of the same authors for a general theory of Tutte groups.

    \subsection{Matroids}
     A \emph{matroid} $M$ is a pair $(E,\mathfrak{I})$, 
where $E$ is a finite \emph{ground set} and $\mathfrak{I}\subseteq2^{E}$ is a 
family of subsets of $E$ that fulfill the following axioms:
\begin{enumerate}[label=(I\arabic{*})]
 \item \label{I1} $\emptyset\in\mathfrak{I}$;
 \item \label{I2} If $I\in\mathfrak{I}$ and $J\subseteq I$, then $J\in\mathfrak{I}$;
 \item \label{I3} If $I$ and $J$ are in $\mathfrak{I}$ and $|I|<|J|$, then there exists $e\in J\setminus I$ with the property 
                  that $I\cup\{e\}\in\mathfrak{I}$.
\end{enumerate}
The subsets of $E$ that belong to $\mathfrak{I}$ are the \emph{independent sets} of $M$. 
Maximal independent sets (with respect to set inclusion) are named \emph{bases} and 
we write $\mathfrak{B}$ for the collection of basis of $M$. 
A \emph{dependent} set of $M$ is a subset of $E$ that is not in $\mathfrak{I}$. 
The subsets of $E$ that are minimal with respect to set inclusion are called \emph{circuits} and 
$\mathfrak{C}$ stands for the family of circuits of $M$.

For a subset $S\subseteq E$, we define its \emph{rank} by the formula
\begin{equation*}
 \operatorname{rk}(S)=\max\left\lbrace
                         |T\cap B|
                         \mid
                         B\in\mathfrak{B}
                        \right\rbrace
\end{equation*}
and we set the \emph{rank of the matroid} $M$ by $\operatorname{rk}(M):=\operatorname{rk}(E)$. 
If $\operatorname{rk}(S)=\operatorname{rk}(M)$ we say that $S$ is a \emph{spanning} set of $M$.

The collection of complements of spanning sets of $M$ does fulfills axioms \ref{I1}, \ref{I2} and \ref{I3}. Hence, it is the 
family of independent sets of a matroid called \emph{dual} to $M$ and denoteb by $M^{\ast}$. Its rank function
$\operatorname{rk}^{\ast}$ is related to that of $M$ 
by the identity 
\begin{equation*}
 \operatorname{rk}^{\ast}(A)=\operatorname{rk}(E\setminus A) + \vert A \vert - \operatorname{rk}(E)
\end{equation*}

We say \emph{cocircuits} and \emph{cobases} of $M$ for the of circuits and bases of $M^{\ast}$.
The collections of cocircuits and cobases of $M$ will be denoted by $\mathfrak{C}^{\ast}$ and $\mathfrak{B}^{\ast}$, 
respectively.

Now, let us consider a subset $T$ of the ground set $E$ of the given matroid $M$.
It is not hard to see that the family of subsets of $T$ that are independent sets of $M$ fulfills axioms \ref{I1}, \ref{I2} 
and \ref{I3}. Thus, this is the collection of independent sets of a matroid named \emph{restriction} of $M$ to $T$.
We write $M[T]$ for the restriction of $M$ to $T$.
On the other hand, the \emph{contraction} of $T$ in $M$ is defined as the matroid $(M^{\ast}[E\setminus T])^{\ast}$. 
A matroid that can be obtained from $M$ with a sequence of restrictions and contractios will be called
a \emph{minor} of $M$.

\begin{example}[The Fano matroid] \label{fano} 
The Fano matroid is the matroid denoted by $F_{7}$ and defined on the ground set $E=\{1,2,\ldots,7\}$ 
by the collection of circuits
\begin{equation*}
 \mathfrak{C}=\left\lbrace
                \{1,2,3\},\{4,5,6\},\{1,4,7\},\{1,5,6\},\{3,4,5\},\{3,6,7\},\{2,4,6\}
              \right\rbrace.
\end{equation*}
Often in this paper 
we will work with matroids ``without minors of Fano or dual-Fano type'', that are matroids for which
neither $F_{7}$ nor its dual can appear as minors.
\end{example}

Given matroids $M_{1}$ and $M_{2}$ with ground sets $E_{1}$ and $E_{2}$ and independent sets 
$\mathfrak{I}_{1}$ and $\mathfrak{I}_{2}$, the \emph{direct sum} of $M_{1}$ and $M_{2}$ is the matroid $M_{1}\oplus M_{2}$ 
with ground set $E_{1}\cup E_{2}$ and independent sets 
\begin{equation*}
 \left\lbrace
 I_{1}\cup I_{2}\mid I_{1}\in\mathfrak{I}_{1}\text{ and }I_{2}\in\mathfrak{I}_{2}
 \right\rbrace
\end{equation*}
We say that $M$ is \emph{disconnected} if there exists a proper non-empty subset $T$ 
of the ground set $E$ such that $M=M[T]\oplus M[E\setminus T]$. $M$ will be called \emph{connected} otherwise.
A \emph{connected component} of $M$ is a maximal inclusion subset $T$ of $E$ such that $M[T]$ is connected.
From \cite[Corollary 4.2.13]{Oxl92} there is a unique (up to permutations) decomposition of $M$ as direct sum of connected 
matroids. In particular, from this it follows that the number $c_{M}$ of connected components of $M$ is well defined.
   
   \subsection{Tutte groups}
    We consider sets of the form $F=C_{1}\cup\cdots\cup C_{k}$ with $C_{i}\in\mathfrak{C}$.
If $\emptyset\subset F_{0}\subset F_{1}\subset\cdots\subset F_{d}=F$
is a maximal chain of such sets, then $d$ depends 
only on $F$ and is called the \emph{dimension} of $F$. We denote it by
$\operatorname{dim}(F)$. Notice that 
\begin{equation*}
 \operatorname{dim}(F)=|F|-\operatorname{rk}(F)-1
\end{equation*}

The group $\mathbb{T}_{M}^{\mathfrak{C},\mathfrak{C}^{\ast}}$ is defined to be the multiplicative abelian group with 
formal generators given by the symbols 
\begin{itemize}
 \item $\epsilon_{M}$;
 \item $C(x)$ for $C\in\mathfrak{C}$ and $x\in C$;
 \item $D(y)$ for $D\in\mathfrak{C}^{\ast}$ and $y\in D$;
\end{itemize}
with relations
\begin{itemize}
 \item $\epsilon_{M}^{2}=1$;
 \item $C(x)D(x)=\epsilon_{M}C(y)D(y)$ for $C\in\mathfrak{C},$ $D\in\mathfrak{C}^{\ast}$ with $\{x,y\}=C\cap D$.
\end{itemize}

The \emph{Tutte group} $\mathbb{T}_{M}$ is then 
the subgroup of $\mathbb{T}_{M}^{\mathfrak{C},\mathfrak{C}^{\ast}}$ generated by 
\begin{itemize}
 \item $\epsilon_{M}$;
 \item $C(x)C(y)^{-1}$ for $C\in\mathfrak{C},$ $x,y\in C$;
 \item $D(x)D(y)^{-1}$ for $D\in\mathfrak{C}^{\ast},$ $x,y\in D$.
\end{itemize}

\begin{definition}[Inner Tutte group]\label{innerdef}
 Let us consider the group homomorphism
 $\Lambda:\mathbb{T}_{M}^{\mathfrak{C},\mathfrak{C}^{\ast}}\longrightarrow
\mathbb{Z}^{|E|}\times\mathbb{Z}^{|\mathfrak{C}|}\times\mathbb{Z}^{|\mathfrak{C}^{\ast}|}$, defined by 
\begin{itemize}
 \item $\epsilon_{M}\mapsto0$;
 \item $C(x)\mapsto(\mathds{1}_{x},\mathds{1}_{C},0)$;
 \item $D(y)\mapsto(-\mathds{1}_{y},0,\mathds{1}_{D})$;
\end{itemize}
where $\mathds{1}$ is the indicator function.
The \emph{inner Tutte group} $\mathbb{T}_{M}^{(0)}$ of $M$ is the kernel of  the homomorphism $\Lambda$. 
\end{definition}

 One can see that 
 $\mathbb{T}_{M}^{(0)}\vartriangleleft\mathbb{T}_{M}\vartriangleleft\mathbb{T}_{M}^{\mathfrak{C},\mathfrak{C}^{\ast}}$. 
 Moreover, if $c_{M}$ is the number of connected components of $M$, with \cite[Theorem 1.5]{DW89} we have
\begin{equation*}
 \begin{array}{ll}
  \mathbb{T}_{M}\cong\mathbb{T}_{M}^{(0)}
                     \times
                     \mathbb{Z}^{|E|-c_{M}} & 
  \mathbb{T}_{M}^{\mathfrak{C},\mathfrak{C}^{\ast}}\cong\mathbb{T}_{M}^{(0)}
                                                        \times
                                                        \mathbb{Z}^{|E|-c_{M}}
                                                        \times
                                                        \mathbb{Z}^{|\mathfrak{C}|}
                                                        \times
                                                        \mathbb{Z}^{|\mathfrak{C}^{\ast}|}
 \end{array}                                                         
\end{equation*}

Thus, any of these groups is known as soon as we know 
$\mathbb{T}_{M}^{(0)}$. In order to prove Theorem \ref{main}, following \cite{GRS95} we provide the subsequent definition.

\begin{definition}\label{TM2}
 Given a matroid $M$ the group $\mathbb{T}_{M}^{(2)}$ is the finetely generated abelian group 
 with formal generators given by the symbols 
\begin{enumerate}[label=(G\arabic{*})]
 \item \label{G1} $\xi_{M}$;
 \item \label{G2} $[C_{i_{1}} C_{i_{2}}|C_{i_{3}} C_{i_{4}}]$, 
                  where $C_{i_{1}}$, $C_{i_{2}}$, $C_{i_{3}}$, $C_{i_{4}}\in\mathfrak{C}$ are circuits of $M$ such that
                  $L=C_{i_{1}}\cup C_{i_{2}}\cup C_{i_{3}}\cup C_{i_{4}}=C_{i_{k}}\cup C_{i_{l}}$ for $k=1,2$, $l=3,4$, 
                  $\operatorname{dim}(L)=1$;
\end{enumerate}
and relations
\begin{enumerate}[label=(R\arabic{*})]
 \item \label{R1} $\xi_{M}^{2}=1_{\mathbb{T}_{M}^{(2)}}$;
 \item \label{R2} $\xi_{M}=1_{\mathbb{T}_{M}^{(2)}}$ if $M$ has minors of Fano or dual-Fano type;
 \item \label{R3} $[C_{i_{1}}C_{i_{2}}|C_{i_{3}}C_{i_{3}}]=1_{\mathbb{T}_{M}^{(2)}}$;
 \item \label{R4} $[C_{i_{1}}C_{i_{2}}|C_{i_{3}}C_{i_{3}}]=[C_{i_{3}}C_{i_{4}}|C_{i_{1}}C_{i_{2}}]$;
 \item \label{R5} $[C_{i_{1}}C_{i_{2}}|C_{i_{3}}C_{i_{4}}]
                  [C_{i_{1}}C_{i_{2}}|C_{i_{4}}C_{i_{5}}]
                  [C_{i_{1}}C_{i_{2}}|C_{i_{5}}C_{i_{3}}]=1_{\mathbb{T}_{M}^{(2)}}$;
 \item \label{R6} $[C_{i_{1}}C_{i_{2}}|C_{i_{3}}C_{i_{4}}]
                  [C_{i_{1}}C_{i_{4}}|C_{i_{2}}C_{i_{3}}]
                  [C_{i_{1}}C_{i_{3}}|C_{i_{4}}C_{i_{2}}]=\xi_{M}$;
 \item \label{R7}$[C_{i_{1}}C_{i_{2}}|C_{i_{6}}C_{i_{9}}]
                  [C_{i_{2}}C_{i_{3}}|C_{i_{4}}C_{i_{7}}]
                  [C_{i_{3}}C_{i_{1}}|C_{i_{5}}C_{i_{8}}]=1_{\mathbb{T}_{M}^{(2)}}$ 
                for any family of circuits $\{C_{i_{1}},\ldots,C_{i_{9}}\}\subseteq\mathfrak{C}$ such that:
               \begin{itemize}
                 \item $\operatorname{dim}(L_{i_{p}})=1$ for $L_{i_p}=C_{i_{q}}\cup C_{i_{r}}$, 
                       where $\{p,q,r\}=\{1,2,3\}$;
                 \item $\operatorname{dim}(P)=2$ where $P=C_{i_{1}}\cup C_{i_{2}}\cup C_{i_{3}}$;
                 \item $C_{i_{s+3}},C_{i_{s+6}}\subseteq L_{i_{s}}$ for $s=1,2,3$;
                 \item $\operatorname{dim}(L_{i_{h}})=1$ for $L_{i_{h}}=C_{i_{3+h}}\cup C_{i_{4+h}}\cup C_{i_{5+h}}$, 
                       $h\in\{1,4\}$;
                 \item $\{C_{i_{1}},C_{i_{2}},C_{i_{3}}\}\cap\{C_{i_{4}},\ldots,C_{i_{9}}\}=\emptyset$.
                \end{itemize}
\end{enumerate}
We write $\mathbb{W}_{M}^{(2)}$ for the set of generators of $\mathbb{T}_{M}^{(2)}$ and we denote by $G_{M}$ the number of 
generators of $\mathbb{T}_{M}^{(2)}$, that is, $G_{M}=\left|\mathbb{W}_{M}^{(2)}\right|$.
Recall that the groups $\mathbb{T}_{M}^{(2)}$ and $\mathbb{T}_{M}^{(0)}$ are isomorphic 
(see \cite[Theorem 4]{GRS95} for more details).
\end{definition}

\begin{note}\label{sym}
Relations \ref{R3} and \ref{R5} imply that 
$[C_{i_{1}}C_{i_{2}}|C_{i_{3}}C_{i_{4}}]=[C_{i_{1}}C_{i_{2}}|C_{i_{4}}C_{i_{3}}]^{-1}$.
\end{note}

 \section{Results}\label{Section2}
   Throughout this section we assume that, for a matroid $M$ with set of circuits $\mathfrak{C}$, an enumeration 
$\{C_{j}\}_{j\in J}$ of $\mathfrak{C}$ and an arbitrary total order $<_{J}$ on $J$ are given.

Our aim is to define a group $\mathcal{T}_{M,<_{J}}^{(0)}$ that is isomorphic to the inner Tutte group of $M$, but it has a 
presentation with less generators than those described in \cite{GRS95}. To show this isomorphism result, we exploit some 
techniques developed in the proof of \cite[Theorem 6.4]{DS15} in order to define group homomorphisms from 
$\mathbb{T}_{M}^{(0)}$ to $\mathcal{T}_{M,<_{J}}^{(0)}$ and from $\mathcal{T}_{M,<_{J}}^{(0)}$ to $\mathbb{T}_{M}^{(0)}$ 
that are one the inverse of the other.

\begin{definition}\label{TMI0}
 We denote by $\mathcal{T}_{M,<_{J}}^{(0)}$ the multiplicative abelian group with formal generators given by the symbols
 \begin{enumerate}[label=(Q\arabic{*})]
  \item \label{Q1} $\eta_{M,<_{J}}$;
  \item \label{Q2} $(C_{j_{1}}C_{j_{2}}|C_{j_{3}}C_{j_{4}})$ 
                   where $C_{j_{1}}$, $C_{j_{2}}$, $C_{j_{3}}$, $C_{j_{4}}\in\mathfrak{C}$ are circuits of $M$ such that
                   $\operatorname{dim}(C_{j_{1}}\cup C_{j_{2}}\cup C_{j_{3}}\cup C_{j_{4}})=1$ 
                   and $j_{1}<j_{2}$, $j_{3}<j_{4}$, $j_{1}<j_{3}$;
 \end{enumerate}
 and relations
  \begin{enumerate}[label=(S\arabic{*})]
   \item \label{S1} $\eta_{M,<_{J}}^{2}=1_{\mathcal{T}_{M,<_{J}}^{(0)}}$;
   \item \label{S2} $\eta_{M,<_{J}}=1_{\mathcal{T}_{M,<_{J}}^{(0)}}$ if $M$ has minors of Fano or dual-Fano type;
   \item \label{S3} $(C_{j_{1}}C_{j_{2}}|C_{j_{3}}C_{j_{4}})
                     (C_{j_{1}}C_{j_{4}}|C_{j_{2}}C_{j_{3}})
                     (C_{j_{1}}C_{j_{3}}|C_{j_{2}}C_{j_{4}})^{-1}=\eta_{M,<_{J}}$ 
                    for any family of circuits $\{C_{j_{1}},C_{j_{2}},C_{j_{3}},C_{j_{4}}\}\subseteq\mathfrak{C}$ such that               
                    $\operatorname{dim}(C_{j_{1}}\cup C_{j_{2}}\cup C_{j_{3}}\cup C_{j_{4}})=1$ 
                    and $j_{1}<j_{2}<j_{3}<j_{4}$;
  \item \label{S4} \begin{equation*}
                      \left\lbrace
                     \begin{array}{l}
                 (C_{j_{1}}C_{j_{2}}|C_{j_{3}}C_{j_{4}})
                 (C_{j_{1}}C_{j_{2}}|C_{j_{4}}C_{j_{5}})
                 (C_{j_{1}}C_{j_{2}}|C_{j_{3}}C_{j_{5}})^{-1}=1_{\mathcal{T}_{M,<_{J}}^{(0)}}\\
                 
                 (C_{j_{1}}C_{j_{3}}|C_{j_{2}}C_{j_{4}})
                 (C_{j_{1}}C_{j_{3}}|C_{j_{4}}C_{j_{5}})
                 (C_{j_{1}}C_{j_{3}}|C_{j_{2}}C_{j_{5}})^{-1}=1_{\mathcal{T}_{M,<_{J}}^{(0)}}\\
                 
                 (C_{j_{1}}C_{j_{4}}|C_{j_{2}}C_{j_{3}})
                 (C_{j_{1}}C_{j_{4}}|C_{j_{3}}C_{j_{5}})
                 (C_{j_{1}}C_{j_{4}}|C_{j_{2}}C_{j_{5}})^{-1}=1_{\mathcal{T}_{M,<_{J}}^{(0)}}\\
                 
                 (C_{j_{1}}C_{j_{5}}|C_{j_{2}}C_{j_{3}})
                 (C_{j_{1}}C_{j_{5}}|C_{j_{3}}C_{j_{4}})
                 (C_{j_{1}}C_{j_{5}}|C_{j_{2}}C_{j_{4}})^{-1}=1_{\mathcal{T}_{M,<_{J}}^{(0)}}\\
                 
                 (C_{j_{1}}C_{j_{4}}|C_{j_{2}}C_{j_{3}})
                 (C_{j_{2}}C_{j_{3}}|C_{j_{4}}C_{j_{5}})
                 (C_{j_{1}}C_{j_{5}}|C_{j_{2}}C_{j_{3}})^{-1}=1_{\mathcal{T}_{M,<_{J}}^{(0)}}\\
                 
                 (C_{j_{1}}C_{j_{3}}|C_{j_{2}}C_{j_{4}})
                 (C_{j_{2}}C_{j_{4}}|C_{j_{3}}C_{j_{5}})
                 (C_{j_{1}}C_{j_{5}}|C_{j_{2}}C_{j_{3}})^{-1}=1_{\mathcal{T}_{M,<_{J}}^{(0)}}\\
                 
                 (C_{j_{1}}C_{j_{3}}|C_{j_{2}}C_{j_{5}})
                 (C_{j_{2}}C_{j_{5}}|C_{j_{3}}C_{j_{4}})
                 (C_{j_{1}}C_{j_{4}}|C_{j_{2}}C_{j_{5}})^{-1}=1_{\mathcal{T}_{M,<_{J}}^{(0)}}\\
                 
                 (C_{j_{1}}C_{j_{2}}|C_{j_{3}}C_{j_{4}})
                 (C_{j_{2}}C_{j_{5}}|C_{j_{3}}C_{j_{4}})
                 (C_{j_{1}}C_{j_{5}}|C_{j_{3}}C_{j_{4}})^{-1}=1_{\mathcal{T}_{M,<_{J}}^{(0)}}\\
                 
                 (C_{j_{1}}C_{j_{2}}|C_{j_{3}}C_{j_{5}})
                 (C_{j_{2}}C_{j_{4}}|C_{j_{3}}C_{j_{5}})
                 (C_{j_{1}}C_{j_{4}}|C_{j_{3}}C_{j_{5}})^{-1}=1_{\mathcal{T}_{M,<_{J}}^{(0)}}\\
                 
                 (C_{j_{1}}C_{j_{2}}|C_{j_{4}}C_{j_{5}})
                 (C_{j_{2}}C_{j_{3}}|C_{j_{4}}C_{j_{5}})
                 (C_{j_{1}}C_{j_{3}}|C_{j_{4}}C_{j_{5}})^{-1}=1_{\mathcal{T}_{M,<_{J}}^{(0)}}\\
               \end{array}
               \right.
              \end{equation*}
              where $C_{j_{1}}$, $C_{j_{2}}$, $C_{j_{3}}$, $C_{j_{4}}$, $C_{j_{5}}\in\mathfrak{C}$ are circuits of $M$ with 
              the properties that
              $\operatorname{dim}(C_{j_{1}}\cup C_{j_{2}}\cup C_{j_{3}}\cup C_{j_{4}}\cup C_{j_{5}})=1$ 
              and $j_{1}<j_{2}<j_{3}<j_{4}<j_{5}$;
  \item \label{S5} $\langle C_{j_{1}}C_{j_{2}}|C_{j_{6}}C_{j_{9}}\rangle
               \langle C_{j_{2}}C_{j_{3}}|C_{j_{4}}C_{j_{7}}\rangle
               \langle C_{j_{3}}C_{j_{1}}|C_{j_{5}}C_{j_{8}}\rangle=1_{\mathcal{T}_{M,<_{J}}^{(0)}}$
              for any family of circuits $\{C_{i_{1}},\ldots,C_{i_{9}}\}\subseteq\mathfrak{C}$ as in \ref{R7} with the 
              extra conditions:
              \begin{enumerate}[label=(O\arabic{*})]
                \item \label{O1} $j_{1}<j_{2}<j_{3}$;
                \item \label{O2} $j_{4}\geq j_{7}$, $j_{5}\geq j_{8}$ and $j_{6}\geq j_{9}$ do not all hold at the same time.
               \end{enumerate}
               Here $\langle C_{d_{1}}C_{d_{2}}|C_{d_{3}}C_{d_{4}}\rangle$ are the symbols given by the formula
               \begin{equation}\tag{$\ast$}\label{FM}
                \langle C_{d_{1}}C_{d_{2}}|C_{d_{3}}C_{d_{4}}\rangle=
                   \left\lbrace
                      \begin{array}{lcccc}
                         1_{\mathcal{T}_{M,<_{J}}^{(0)}} & \text{if} & d_{1}=d_{2} & \text{or}   & d_{3}=d_{4}\\
            (C_{d_{1}}C_{d_{2}}|C_{d_{3}}C_{d_{4}})      & \text{if} & d_{1}<d_{2} & d_{3}<d_{4} & d_{1}<d_{3} \\
            (C_{d_{3}}C_{d_{4}}|C_{d_{1}}C_{d_{2}})      & \text{if} & d_{1}<d_{2} & d_{3}<d_{4} & d_{3}<d_{1} \\
            (C_{d_{1}}C_{d_{2}}|C_{d_{4}}C_{d_{3}})^{-1} & \text{if} & d_{1}<d_{2} & d_{4}<d_{3} & d_{1}<d_{4} \\
            (C_{d_{4}}C_{d_{3}}|C_{d_{1}}C_{d_{2}})^{-1} & \text{if} & d_{1}<d_{2} & d_{4}<d_{3} & d_{4}<d_{1} \\
            (C_{d_{2}}C_{d_{1}}|C_{d_{3}}C_{d_{4}})^{-1} & \text{if} & d_{2}<d_{1} & d_{3}<d_{4} & d_{2}<d_{3} \\
            (C_{d_{3}}C_{d_{4}}|C_{d_{2}}C_{d_{1}})^{-1} & \text{if} & d_{2}<d_{1} & d_{3}<d_{4} & d_{3}<d_{2} \\
            (C_{d_{2}}C_{d_{1}}|C_{d_{4}}C_{d_{3}})      & \text{if} & d_{2}<d_{1} & d_{4}<d_{3} & d_{2}<d_{4} \\
            (C_{d_{4}}C_{d_{3}}|C_{d_{2}}C_{d_{1}})      & \text{if} & d_{2}<d_{1} & d_{4}<d_{3} & d_{4}<d_{2} \\
                      \end{array}
                   \right.
               \end{equation}
               and defined for any family of circuits $\{C_{d_{1}},C_{d_{2}},C_{d_{3}},C_{d_{1}}\}\subseteq\mathfrak{C}$
               satisfying $\operatorname{dim}(C_{d_{1}}\cup C_{d_{2}}\cup C_{d_{3}}\cup C_{d_{4}})=1$ and 
               $\{C_{d_{1}},C_{d_{2}}\}\cap\{C_{d_{3}},C_{d_{4}}\}=\emptyset$.
  \end{enumerate}
We write $\mathcal{W}_{M,<_{J}}^{(0)}$ for the set of generators of $\mathcal{T}_{M,<_{J}}^{(0)}$ and we denote by 
$g_{M}$ the number of generators of $\mathcal{T}_{M,<_{H}}^{(0)}$, 
that is, $g_{M}=\left|\mathcal{W}_{M,<_{J}}^{(0)}\right|$.
\end{definition}

\begin{note}
 If $\{C_{h}\}_{h\in H}$ is another enumeration of $\mathfrak{C}$ with total order $<_{H}$ on $H$, 
 then the groups $\mathcal{T}_{M,<_{H}}^{(0)}$ and $\mathcal{T}_{M,<_{J}}^{(0)}$ are isomorphic. To see this it suffices to 
 consider a suitable relabeling of the circuits of $M$. As a consequence of this, the quantity $g_{M}$ is 
 well defined, neither depending on the choice of the enumeration of circuits of $M$ 
 nor the total ordering of such enumeration. 
\end{note}

\begin{theorem}\label{main}
 The groups $\mathcal{T}_{M,<_{J}}^{(0)}$ and $\mathbb{T}_{M}^{(0)}$ are isomorphic. 
\end{theorem}

In Table \ref{table1} we compare the number $G_{M}$ of generators of the presentation of the inner Tutte group 
of \cite[Theorem 4]{GRS95} (see Definition \ref{TM2}) with 
the number $g_{M}$ of generators of our presentation (see Definition \ref{TMI0}) for some remarkable matroids.
These results are obtained with SAGE on a standard laptop.

\begin{table}
\begin{equation*}
 \begin{array}{lcc}
  \toprule
  \text{Matroid $M$}      & G_{M} & g_{M}  \\
  \midrule
   U_{2}(4)               & 85 &   4 \\
  \midrule
   U_{2}(5)               & 421 & 16 \\
   \midrule
   U_{3}(5)               & 261 & 16 \\
   \midrule
   M(K_{4})               & 109 & 1   \\
   \midrule
   \mathcal{W}^{3}        & 307 & 10  \\
   \midrule
   Q_{6}                  & 615 & 28  \\
   \midrule
   P_{6}                  & 1033 & 55 \\
   \midrule 
   U_{3}(6)               & 1561 & 91 \\
   \midrule 
   R_{6}                  & 505  & 19 \\
   \midrule 
   F_{7}                  & 379  & 1  \\
   \midrule
   F_{7}^{\ast}           & 127  & 1  \\
   \midrule
   F_{7}^{-}              & 775  & 19 \\
   \midrule
   (F_{7}^{-})^{\ast}     & 325  & 10 \\
   \midrule 
   P_{7}                  & 1171 & 37 \\
 \bottomrule
 \end{array}
\end{equation*}
\caption{Comparison between $G_{M}$ and $g_{M}$ for some remarkable cases.\label{table1}}
\end{table}

\begin{proof}[{Proof of Theorem \ref{main}}]
 With \cite[Theorem 4]{GRS95} it is enough to prove that the groups $\mathbb{T}_{M}^{(2)}$ and $\mathcal{T}_{M,<_{J}}^{(0)}$ 
 are isomorphic. To see this, let $\mathbb{W}_{M}^{(2)}$ and $\mathcal{W}_{M,<_{J}}^{(0)}$ be the set of generators of 
 $\mathbb{T}_{M}^{(2)}$ and $\mathcal{T}_{M,<_{J}}^{(0)}$ as in Definition \ref{TM2} and Definition \ref{TMI0}.
 Let us consider the map $\phi:\mathbb{W}_{M}^{(2)}\longrightarrow\mathcal{T}_{M,<_{J}}^{(0)}$ given by
 \begin{enumerate}[label=(D\arabic{*})]
  \item \label{D1} $\phi(\xi_{M})=\eta_{M,<_{J}}$;
  
  \item \label{D2} $\phi([C_{i_{1}}C_{i_{2}}|C_{i_{3}}C_{i_{4}}])=\langle C_{i_{1}}C_{i_{2}}|C_{i_{3}}C_{i_{4}}\rangle$ where 
                   $\langle C_{i_{1}}C_{i_{2}}|C_{i_{3}}C_{i_{4}}\rangle$ are the symbols defined in \eqref{FM}.
 \end{enumerate}
 
 The map $\phi:\mathbb{W}_{M}^{(2)}\longrightarrow\mathcal{T}_{M,<_{J}}^{(0)}$ fulfills the subsequent properties:
 \begin{enumerate}[label=(VR\arabic{*})]
  
  \item \label{VR1}
        $\phi(\xi_{M})^{2}=1_{\mathcal{T}_{M,<_{J}}^{(0)}}$. This follows directly from \ref{D1} and \ref{S1}.
  
  \item \label{VR2}
        $\phi(\xi_{M})=1_{\mathcal{T}_{M,<_{J}}^{(0)}}$ if $M$ has minors of Fano or dual-Fano type. This follows from 
        \ref{D1} and \ref{S2}.
  
  \item \label{VR3}
        $\phi([C_{i_{1}}C_{i_{2}}|C_{i_{3}}C_{i_{3}}])=1_{\mathcal{T}_{M,<_{J}}^{(0)}}$. This follows immediately from 
        \eqref{FM}.
        
  \item \label{VR4}
        $\phi([C_{i_{1}}C_{i_{2}}|C_{i_{3}}C_{i_{4}}])=\phi([C_{i_{3}}C_{i_{4}}|C_{i_{1}}C_{i_{2}}])$. This follows from a
        straightforward check of \eqref{FM}.

  \item \label{VR5}
        $\phi([C_{i_{1}}C_{i_{2}}|C_{i_{3}}C_{i_{4}}])
         \phi([C_{i_{1}}C_{i_{2}}|C_{i_{4}}C_{i_{5}}])
         \phi([C_{i_{1}}C_{i_{2}}|C_{i_{5}}C_{i_{3}}])=1_{\mathcal{T}_{M,<_{J}}^{(0)}}$.
        To see this, we need to distinguish between two cases:
        \begin{itemize}
        
         \item If $i_{1}=i_{2}$  or $i_{3}=i_{4}$ or $i_{4}=i_{5}$ or $i_{3}=i_{5}$ 
               \ref{VR5} easly follows from \eqref{FM}.
            
         \item Assume $i_{1}\neq i_{2}$, $i_{3}\neq i_{4}$, $i_{4}\neq i_{5}$, $i_{3}\neq i_{5}$. 
               From \eqref{FM}, it is not hard to see that \ref{VR5} is equivalent to exaclty one among the equations of 
               the family \ref{S4}. Thus, \ref{VR5} fulfills since all the equations of the family \ref{S4} hold.
       \end{itemize}
  
  \item \label{VR6}
        $\phi([C_{i_{1}}C_{i_{2}}|C_{i_{3}}C_{i_{4}}])
         \phi([C_{i_{1}}C_{i_{4}}|C_{i_{2}}C_{i_{3}}])
         \phi([C_{i_{1}}C_{i_{3}}|C_{i_{4}}C_{i_{2}}])=\eta_{M,<_{J}}$. To check this, let 
        $C_{i_{1}}$, $C_{i_{2}}$, $C_{i_{3}}$, $C_{i_{4}}$ be pairwise distinct circuits of $M$ such that 
        $\operatorname{dim}(C_{i_{1}}\cup C_{i_{2}}\cup C_{i_{3}}\cup C_{i_{4}})=1$ and let $\sigma\in\mathcal{S}_{4}$ be a 
        permutation such that $i_{\sigma(1)}<i_{\sigma(2)}<i_{\sigma(3)}<i_{\sigma(4)}$. 
        From \ref{VR3}, \ref{VR4} and \ref{VR5}, together with Note \ref{sym}, the symbols
        \begin{equation*}
         \phi([C_{i_{1}}C_{i_{2}}|C_{i_{3}}C_{i_{4}}])
        \end{equation*}
        fulfill the hypothesis of Lemma \ref{A1}. Thus, \ref{VR6} is equivalent to
         \begin{equation*}
         \begin{aligned}
                & \phi([C_{i_{\sigma(1)}}C_{i_{\sigma(2)}}|C_{i_{\sigma(3)}}C_{i_{\sigma(4)}}])                \\
          \cdot & \phi([C_{i_{\sigma(1)}}C_{i_{\sigma(4)}}|C_{i_{\sigma(2)}}C_{i_{\sigma(3)}}])                \\
          \cdot & \phi([C_{i_{\sigma(1)}}C_{i_{\sigma(3)}}|C_{i_{\sigma(4)}}C_{i_{\sigma(2)}}])=\eta_{M,<_{J}} \\
         \end{aligned}
         \end{equation*}
         From \eqref{FM}, this is the same as
         \begin{equation*}
         \begin{aligned}
                & (C_{i_{\sigma(1)}}C_{i_{\sigma(2)}}|C_{i_{\sigma(3)}}C_{i_{\sigma(4)}})                      \\
          \cdot & (C_{i_{\sigma(1)}}C_{i_{\sigma(4)}}|C_{i_{\sigma(2)}}C_{i_{\sigma(3)}})                      \\
          \cdot & (C_{i_{\sigma(1)}}C_{i_{\sigma(3)}}|C_{i_{\sigma(2)}}C_{i_{\sigma(4)}})^{-1}=\eta_{M,<_{J}}  \\
         \end{aligned}
         \end{equation*}
         and this holds by \ref{S3}.
   
  \item \label{VR7}
         $\phi([C_{i_{1}}C_{i_{2}}|C_{i_{6}}C_{i_{9}}])
          \phi([C_{i_{2}}C_{i_{3}}|C_{i_{4}}C_{i_{7}}])
          \phi([C_{i_{3}}C_{i_{1}}|C_{i_{5}}C_{i_{8}}])=1_{\mathcal{T}_{M,<_{J}}^{(0)}}$
         for any family of circuits $C_{i_{1}},\ldots,C_{i_{9}}$ as in \ref{R7}. 
         To check this, we have to distinguish between two subcases.
         \begin{itemize}
          \item If $i_{6}=i_{9}$, $i_{4}=i_{7}$ and $i_{5}=i_{8}$, 
                \ref{VR7} is obviously verified since \ref{VR3} holds.
                
          \item Conversely, there is a permutation $\sigma\in\mathcal{S}_{9}$ such that conditions \ref{O1} and \ref{O2} of 
                Definition \ref{TMI0} fulfill with $j_{d}=i_{\sigma(d)}$, $1\leq d\leq 9$. 
                With the same arguments of the proof of \ref{VR6}, we can apply Lemma \ref{A2}. 
                Hence, \ref{VR7} is equivalent to
                \begin{equation*}
                 \begin{aligned}
                  & \phi([C_{i_{\sigma(1)}}C_{i_{\sigma(2)}}|
                          C_{i_{\sigma(\sigma^{-1}(3)+3)}}C_{i_{\sigma(\sigma^{-1}(3)+6)}}])\cdot \\
          \cdot   & \phi([C_{i_{\sigma(2)}}C_{i_{\sigma(3)}}|
                          C_{i_{\sigma(\sigma^{-1}(1)+3)}}C_{i_{\sigma(\sigma^{-1}(1)+6)}}])\cdot \\ 
          \cdot   & \phi([C_{i_{\sigma(3)}}C_{i_{\sigma(1)}}|
                          C_{i_{\sigma(\sigma^{-1}(2)+3)}}C_{j_{\sigma(\sigma^{-1}(2)+6)}}]                    
                      =1_{\mathcal{T}_{M,<_{J}}^{(0)}}                                               \\
                 \end{aligned}
                \end{equation*}
                From \eqref{FM}, this is the same as
                \begin{equation*}
                 \begin{aligned}
        & \langle C_{i_{\sigma(1)}}C_{i_{\sigma(2)}}|
                  C_{i_{\sigma(\sigma^{-1}(3)+3)}}C_{i_{\sigma(\sigma^{-1}(3)+6)}}\rangle\cdot \\
\cdot   & \langle C_{i_{\sigma(2)}}C_{i_{\sigma(3)}}|
                  C_{i_{\sigma(\sigma^{-1}(1)+3)}}C_{i_{\sigma(\sigma^{-1}(1)+6)}}\rangle\cdot \\
\cdot   & \langle C_{i_{\sigma(3)}}C_{i_{\sigma(1)}}|
                  C_{i_{\sigma(\sigma^{-1}(2)+3)}}C_{j_{\sigma(\sigma^{-1}(2)+6)}}\rangle                    
                =1_{\mathcal{T}_{M,<_{J}}^{(0)}}                                               \\
                 \end{aligned}       
                \end{equation*}
                and this holds by \ref{S5}.
         \end{itemize}
 \end{enumerate}
 
 So that, there exists a unique group homomorphism 
 $\Phi:\mathbb{T}_{M}^{(2)}\longrightarrow\mathcal{T}_{M,<_{J}}^{(0)}$ with $\left.\Phi\right|_{\mathbb{W}_{M}^{(2)}}=\phi$.
 Similarly, let us consider the map $\psi:\mathcal{W}_{M,<_{J}}^{(0)}\longrightarrow\mathbb{T}_{M}^{(2)}$ defined by
 \begin{enumerate}[label=(T\arabic{*})]
  \item \label{T1} $\psi(\eta_{M,<_{J}})=\xi_{M}$;
  
  \item \label{T2} $\psi((C_{i_{1}}C_{i_{2}}|C_{i_{3}}C_{i_{4}}))=[C_{i_{1}}C_{i_{2}}|C_{i_{3}}C_{i_{4}}]$.
 \end{enumerate}
 
 The map $\psi:\mathcal{W}_{M,<_{J}}^{(0)}\longrightarrow\mathbb{T}_{M}^{(2)}$ satisfies the subsequent properties:
 \begin{enumerate}[label=(VS\arabic{*})]
  \item \label{VS1}
        $\psi(\eta_{M,<_{J}})^{2}=1_{\mathbb{T}_{M}^{(2)}}$. This follows from \ref{T1} and \ref{R1}.
  
  \item \label{VS2}
        $\psi(\eta_{M,<_{J}})=1_{\mathbb{T}_{M}^{(2)}}$ if $M$ has minors of Fano or dual-Fano type. 
        This follows from \ref{T1} together with \ref{R2}.
        
  \item \label{VS3}
        $\psi((C_{i_{1}}C_{i_{2}}|C_{i_{3}}C_{i_{4}}))
         \psi((C_{i_{1}}C_{i_{4}}|C_{i_{2}}C_{i_{3}}))
         \psi((C_{i_{1}}C_{i_{3}}|C_{i_{2}}C_{i_{4}}))^{-1}=\xi_{M}$ 
         for any
         family $\{C_{j_{1}},C_{j_{2}},C_{j_{3}},C_{j_{4}}\}\subseteq\mathfrak{C}$ with               
         $\operatorname{dim}(C_{j_{1}}\cup C_{j_{2}}\cup C_{j_{3}}\cup C_{j_{4}})=1$ and $j_{1}<j_{2}<j_{3}<j_{4}$.
         To see this, notice that by \ref{T2}, \ref{VS3} is equivalent to
         \begin{equation*}
          [C_{i_{1}}C_{i_{2}}|C_{i_{3}}C_{i_{4}}]
          [C_{i_{1}}C_{i_{4}}|C_{i_{2}}C_{i_{3}}]
          [C_{i_{1}}C_{i_{3}}|C_{i_{2}}C_{i_{4}}]^{-1}=\xi_{M}
         \end{equation*}
         By Note \ref{sym} this 
         is the same as 
         \begin{equation*}
          [C_{i_{1}}C_{i_{2}}|C_{i_{3}}C_{i_{4}}]
          [C_{i_{1}}C_{i_{4}}|C_{i_{2}}C_{i_{3}}]
          [C_{i_{1}}C_{i_{3}}|C_{i_{4}}C_{i_{2}}]=\xi_{M}
         \end{equation*}
         and this holds by \ref{R6}.
         
  \item \label{VS4} 
          \begin{equation*}
               \left\lbrace
                \begin{array}{l}
                  \psi((C_{j_{1}}C_{j_{2}}|C_{j_{3}}C_{j_{4}}))
                  \psi((C_{j_{1}}C_{j_{2}}|C_{j_{4}}C_{j_{5}}))
                  \psi((C_{j_{1}}C_{j_{2}}|C_{j_{3}}C_{j_{5}}))^{-1}=1_{\mathbb{T}_{M}^{(2)}} \\
                  
                  \psi((C_{j_{1}}C_{j_{3}}|C_{j_{2}}C_{j_{4}}))
                  \psi((C_{j_{1}}C_{j_{3}}|C_{j_{4}}C_{j_{5}}))
                  \psi((C_{j_{1}}C_{j_{3}}|C_{j_{2}}C_{j_{5}}))^{-1}=1_{\mathbb{T}_{M}^{(2)}} \\
                  
                  \psi((C_{j_{1}}C_{j_{4}}|C_{j_{2}}C_{j_{3}}))
                  \psi((C_{j_{1}}C_{j_{4}}|C_{j_{3}}C_{j_{5}}))
                  \psi((C_{j_{1}}C_{j_{4}}|C_{j_{2}}C_{j_{5}}))^{-1}=1_{\mathbb{T}_{M}^{(2)}} \\
                  
                  \psi((C_{j_{1}}C_{j_{5}}|C_{j_{2}}C_{j_{3}}))
                  \psi((C_{j_{1}}C_{j_{5}}|C_{j_{3}}C_{j_{4}}))
                  \psi((C_{j_{1}}C_{j_{5}}|C_{j_{2}}C_{j_{4}}))^{-1}=1_{\mathbb{T}_{M}^{(2)}} \\
                  
                  \psi((C_{j_{1}}C_{j_{4}}|C_{j_{2}}C_{j_{3}}))
                  \psi((C_{j_{2}}C_{j_{3}}|C_{j_{4}}C_{j_{5}}))
                  \psi((C_{j_{1}}C_{j_{5}}|C_{j_{2}}C_{j_{3}}))^{-1}=1_{\mathbb{T}_{M}^{(2)}} \\
                  
                  \psi((C_{j_{1}}C_{j_{3}}|C_{j_{2}}C_{j_{4}}))
                  \psi((C_{j_{2}}C_{j_{4}}|C_{j_{3}}C_{j_{5}}))
                  \psi((C_{j_{1}}C_{j_{5}}|C_{j_{2}}C_{j_{3}}))^{-1}=1_{\mathbb{T}_{M}^{(2)}} \\
                  
                  \psi((C_{j_{1}}C_{j_{3}}|C_{j_{2}}C_{j_{5}}))
                  \psi((C_{j_{2}}C_{j_{5}}|C_{j_{3}}C_{j_{4}}))
                  \psi((C_{j_{1}}C_{j_{4}}|C_{j_{2}}C_{j_{5}}))^{-1}=1_{\mathbb{T}_{M}^{(2)}} \\
                  
                  \psi((C_{j_{1}}C_{j_{2}}|C_{j_{3}}C_{j_{4}}))
                  \psi((C_{j_{2}}C_{j_{5}}|C_{j_{3}}C_{j_{4}}))
                  \psi((C_{j_{1}}C_{j_{5}}|C_{j_{3}}C_{j_{4}}))^{-1}=1_{\mathbb{T}_{M}^{(2)}} \\
                  
                  \psi((C_{j_{1}}C_{j_{2}}|C_{j_{3}}C_{j_{5}}))
                  \psi((C_{j_{2}}C_{j_{4}}|C_{j_{3}}C_{j_{5}}))
                  \psi((C_{j_{1}}C_{j_{4}}|C_{j_{3}}C_{j_{5}}))^{-1}=1_{\mathbb{T}_{M}^{(2)}} \\
                  
                  \psi((C_{j_{1}}C_{j_{2}}|C_{j_{4}}C_{j_{5}}))
                  \psi((C_{j_{2}}C_{j_{3}}|C_{j_{4}}C_{j_{5}}))
                  \psi((C_{j_{1}}C_{j_{3}}|C_{j_{4}}C_{j_{5}}))^{-1}=1_{\mathbb{T}_{M}^{(2)}} \\
                \end{array}
               \right.
              \end{equation*}
              where $C_{j_{1}}$, $C_{j_{2}}$, $C_{j_{3}}$, $C_{j_{4}}$, $C_{j_{5}}\in\mathfrak{C}$ are circuits of $M$ with 
              the properties that
              $\operatorname{dim}(C_{j_{1}}\cup C_{j_{2}}\cup C_{j_{3}}\cup C_{j_{4}}\cup C_{j_{5}})=1$ 
              and $j_{1}<j_{2}<j_{3}<j_{4}<j_{5}$.
              To see this, notice that by \ref{T2}, together with \ref{R4} and Note \ref{sym}, 
              the given family of equations
              is equivalent to 
              \begin{equation*}
               \left\lbrace
               \begin{array}{l}
                 [C_{j_{1}}C_{j_{2}}|C_{j_{3}}C_{j_{4}}]
                 [C_{j_{1}}C_{j_{2}}|C_{j_{4}}C_{j_{5}}]
                 [C_{j_{1}}C_{j_{2}}|C_{j_{5}}C_{j_{3}}]=1_{\mathbb{T}_{M}^{(2)}}\\
                 
                 [C_{j_{1}}C_{j_{3}}|C_{j_{2}}C_{j_{4}}]
                 [C_{j_{1}}C_{j_{3}}|C_{j_{4}}C_{j_{5}}]
                 [C_{j_{1}}C_{j_{3}}|C_{j_{5}}C_{j_{2}}]=1_{\mathbb{T}_{M}^{(2)}}\\
                 
                 [C_{j_{1}}C_{j_{4}}|C_{j_{2}}C_{j_{3}}]
                 [C_{j_{1}}C_{j_{4}}|C_{j_{3}}C_{j_{5}}]
                 [C_{j_{1}}C_{j_{4}}|C_{j_{5}}C_{j_{2}}]=1_{\mathbb{T}_{M}^{(2)}}\\
                 
                 [C_{j_{1}}C_{j_{5}}|C_{j_{2}}C_{j_{3}}]
                 [C_{j_{1}}C_{j_{5}}|C_{j_{3}}C_{j_{4}}]
                 [C_{j_{1}}C_{j_{5}}|C_{j_{3}}C_{j_{2}}]=1_{\mathbb{T}_{M}^{(2)}}\\
                 
                 [C_{j_{2}}C_{j_{3}}|C_{j_{1}}C_{j_{4}}]
                 [C_{j_{2}}C_{j_{3}}|C_{j_{4}}C_{j_{5}}]
                 [C_{j_{2}}C_{j_{3}}|C_{j_{5}}C_{j_{1}}]=1_{\mathbb{T}_{M}^{(2)}}\\
                 
                 [C_{j_{2}}C_{j_{4}}|C_{j_{1}}C_{j_{3}}]
                 [C_{j_{2}}C_{j_{4}}|C_{j_{3}}C_{j_{5}}]
                 [C_{j_{2}}C_{j_{4}}|C_{j_{5}}C_{j_{1}}]=1_{\mathbb{T}_{M}^{(2)}}\\
                 
                 [C_{j_{2}}C_{j_{5}}|C_{j_{1}}C_{j_{3}}]
                 [C_{j_{2}}C_{j_{5}}|C_{j_{3}}C_{j_{4}}]
                 [C_{j_{2}}C_{j_{5}}|C_{j_{4}}C_{j_{1}}]=1_{\mathbb{T}_{M}^{(2)}}\\
                 
                 [C_{j_{3}}C_{j_{4}}|C_{j_{1}}C_{j_{2}}]
                 [C_{j_{3}}C_{j_{4}}|C_{j_{2}}C_{j_{5}}]
                 [C_{j_{3}}C_{j_{4}}|C_{j_{5}}C_{j_{1}}]=1_{\mathbb{T}_{M}^{(2)}}\\
                 
                 [C_{j_{3}}C_{j_{5}}|C_{j_{1}}C_{j_{2}}]
                 [C_{j_{3}}C_{j_{5}}|C_{j_{2}}C_{j_{4}}]
                 [C_{j_{3}}C_{j_{5}}|C_{j_{4}}C_{j_{1}}]=1_{\mathbb{T}_{M}^{(2)}}\\
                 
                 [C_{j_{4}}C_{j_{5}}|C_{j_{1}}C_{j_{2}}]
                 [C_{j_{4}}C_{j_{5}}|C_{j_{2}}C_{j_{3}}]
                 [C_{j_{4}}C_{j_{5}}|C_{j_{3}}C_{j_{1}}]=1_{\mathbb{T}_{M}^{(2)}}\\
               \end{array}
               \right.
              \end{equation*}
              and all of these identities hold by \ref{R5}.
              
   \item \label{VS5}
         $\psi(\langle C_{j_{1}}C_{j_{2}}|C_{j_{6}}C_{j_{9}}\rangle)
          \psi(\langle C_{j_{2}}C_{j_{3}}|C_{j_{4}}C_{j_{7}}\rangle)
          \psi(\langle C_{j_{3}}C_{j_{1}}|C_{j_{5}}C_{j_{8}}\rangle)=1_{\mathbb{T}_{M}^{(2)}}$
         for any family of circuits $\{C_{i_{1}},\ldots,C_{i_{9}}\}\subseteq\mathfrak{C}$ as in \ref{S5}. 
         To check this, notice that \ref{T2} and \eqref{FM}, together with \ref{R4} and Note \ref{sym}, imply that  
          \begin{equation*}
           \psi(\langle C_{d_{1}}C_{d_{2}}|C_{d_{3}}C_{d_{4}}\rangle)=[ C_{d_{1}}C_{d_{2}}|C_{d_{3}}C_{d_{4}}]
          \end{equation*}
         Thus, \ref{VS5} is equivalent to
         \begin{equation*}
          [C_{j_{1}}C_{j_{2}}|C_{j_{6}}C_{j_{9}}]
          [C_{j_{2}}C_{j_{3}}|C_{j_{4}}C_{j_{7}}]
          [C_{j_{3}}C_{j_{1}}|C_{j_{5}}C_{j_{8}}]=1_{\mathbb{T}_{M}^{(2)}}
         \end{equation*}
         and this holds by \ref{R7}.
 \end{enumerate}
 
 Hence, there is a unique homomorphism 
 $\Psi:\mathcal{T}_{M,<_{J}}^{(0)}\longrightarrow\mathbb{T}_{M}^{(2)}$ with 
 $\left.\Psi\right|_{\mathcal{W}_{M,<_{J}}^{(0)}}=\psi$.
 To complete our proof it is enough to show that $\Psi$ is bijective. Exploting relations \ref{R3} and \ref{R4}, 
 together with Note \ref{sym}, we can see that $\mathbb{T}_{M}^{(2)}$ is actually generated by those symbols 
 $[C_{i_{1}} C_{i_{2}}|C_{i_{3}} C_{i_{4}}]$ where $i_{1}<i_{2}<i_{3}<i_{4}$. 
 In particular, this implies that $\Psi$ is surjective. On the other hand, by definition of the maps $\phi$ and $\psi$ 
 we have $\phi\circ\psi=\operatorname{Id}_{\mathcal{W}_{M,<_{J}}^{(0)}}$. 
 So that, $\Phi\circ\Psi=\operatorname{Id}_{\mathcal{T}_{M,<_{J}}^{(0)}}$ and from this it follows that $\Psi$ is injective 
 as well.
\end{proof}

  \appendix
    \section{Technical proofs}\label{AppendixA}
     For a given matroid $M$ with circuit set $\mathfrak{C}$, we consider $\mathbb{T}_{M}^{(2)}\text{-values}$
\begin{equation*}
 \lceil C_{q_{1}}C_{q_{2}}|C_{q_{3}}C_{q_{4}}\rceil
\end{equation*}
defined for circuits $C_{q_{1}}$, $C_{q_{2}}$, $C_{q_{3}}$, $C_{q_{4}}\in\mathfrak{C}$ such that 
$\operatorname{dim}(C_{q_{1}}\cup C_{q_{2}}\cup C_{q_{3}}\cup C_{q_{4}})=1$ and
$\{C_{q_{1}},C_{q_{2}}\}\cap\{C_{q_{3}},C_{q_{4}}\}=\emptyset$  
and fulfilling
\begin{enumerate}[label=(P\arabic{*})]
 \item \label{P1} $\lceil C_{q_{1}}C_{q_{2}}|C_{q_{3}}C_{q_{3}}\rceil=1_{\mathbb{T}_{M}^{(2)}}$
 \item \label{P2} $\lceil C_{q_{1}}C_{q_{2}}|C_{q_{3}}C_{q_{4}}\rceil=\lceil C_{q_{3}}C_{q_{4}}|C_{q_{1}}C_{q_{2}}\rceil$
 \item \label{P3} $\lceil C_{q_{1}}C_{q_{2}}|C_{q_{3}}C_{q_{4}}\rceil=\lceil C_{q_{1}}C_{q_{2}}|C_{q_{4}}C_{q_{3}}\rceil^{-1}$
\end{enumerate}

\begin{lemma}\label{A1}
 Let us consider a permutation 
 \begin{equation*}
  \sigma=\left(\begin{array}{cccc}
                     1    &     2     &     3     &     4     \\
                \sigma(1) & \sigma(2) & \sigma(3) & \sigma(4) \\
               \end{array}
         \right)\in\mathcal{S}_{4}.
 \end{equation*} 
 Then,
 \begin{equation*}
  \lceil C_{q_{1}}C_{q_{2}}|C_{q_{3}}C_{q_{4}}\rceil
  \lceil C_{q_{1}}C_{q_{4}}|C_{q_{2}}C_{q_{3}}\rceil
  \lceil C_{q_{1}}C_{q_{3}}|C_{q_{4}}C_{q_{2}}\rceil=1_{\mathbb{T}_{M}^{(2)}}
 \end{equation*}
 is equivalent to
 \begin{equation*}
  \begin{aligned}
         & \lceil C_{q_{\sigma(1)}}C_{q_{\sigma(2)}}|C_{q_{\sigma(3)}}C_{q_{\sigma(4)}}\rceil                          \\
   \cdot & \lceil C_{q_{\sigma(1)}}C_{q_{\sigma(4)}}|C_{q_{\sigma(2)}}C_{q_{\sigma(3)}}\rceil                          \\
   \cdot & \lceil C_{q_{\sigma(1)}}C_{q_{\sigma(3)}}|C_{q_{\sigma(4)}}C_{q_{\sigma(2)}}\rceil=1_{\mathbb{T}_{M}^{(2)}} \\
  \end{aligned}
 \end{equation*}
\end{lemma}

\begin{proof}
 Since $\mathcal{S}_{4}$ is generated by transpositions $(12)$, $(23)$, $(34)$ it is 
 enough to verify our statement for $\sigma$ being such transposition, which can be done by a direct computation using 
 \ref{P1}, \ref{P2} and \ref{P3}.
\end{proof}

\begin{lemma}\label{A2}
 Let $\{C_{q_{1}},\ldots,C_{q_{9}}\}\subseteq\mathfrak{C}$ be circuits as in \ref{R7} and let $\iota$, $\nu\in\mathcal{S}_{9}$ 
 be the permutations
 \begin{equation*}
  \iota=\left(\begin{array}{ccccccccc}
                 1    &   2      &  3       &  4     &   5    &   6    &    7   &    8   & 9 \\
             \iota(1) & \iota(2) & \iota(3) &  4     &   5    &   6    &    7   &    8   & 9 \\
            \end{array}
      \right)
 \end{equation*}
 and 
 \begin{equation*}
  \nu=(47)(58)(69)
 \end{equation*}
 Finally, let 
 \begin{equation*}
 \sigma=\left(\begin{array}{ccccccccc}
            1  &   2       &  3        &  4        &   5       &   6       &    7      &    8      & 9 \\
     \sigma(1) & \sigma(2) & \sigma(3) & \sigma(4) & \sigma(5) & \sigma(6) & \sigma(7) & \sigma(8) & \sigma(9) \\
            \end{array}
      \right)\in\langle\mu,\nu\rangle
\end{equation*}
Then, 
\begin{equation*}
 \lceil C_{q_{1}}C_{q_{2}}|C_{q_{6}}C_{q_{9}}\rceil
 \lceil C_{q_{1}}C_{q_{2}}|C_{q_{6}}C_{q_{9}}\rceil
 \lceil C_{q_{1}}C_{q_{2}}|C_{q_{6}}C_{q_{9}}\rceil=1_{\mathbb{T}_{M}^{(2)}} 
\end{equation*}
is equivalent to
\begin{equation*}
 \begin{aligned}
              & \lceil C_{q_{\sigma(1)}}C_{q_{\sigma(2)}}|
                       C_{q_{\sigma(\sigma^{-1}(3)+3)}}C_{q_{\sigma(\sigma^{-1}(3)+6)}}\rceil\cdot \\
      \cdot   & \lceil C_{q_{\sigma(2)}}C_{q_{\sigma(3)}}|
                       C_{q_{\sigma(\sigma^{-1}(1)+3)}}C_{q_{\sigma(\sigma^{-1}(1)+6)}}\rceil\cdot \\
      \cdot   & \lceil C_{q_{\sigma(3)}}C_{q_{\sigma(1)}}|
                       C_{q_{\sigma(\sigma^{-1}(2)+3)}}C_{q_{\sigma(\sigma^{-1}(2)+6)}}\rceil                    
                =1_{\mathbb{T}_{M}^{(2)}}                                                   \\
                 \end{aligned} 
\end{equation*}
\end{lemma}

\begin{proof}
 The subgroup $\lbrace\iota,\nu\rbrace<\mathcal{S}_{9}$ is generated by $(12)$, $(23)$ and $\nu$.
 Hence, it suffices to check the statement for these permutations. 
 With \ref{P2} and \ref{P3}, this can be easly done by direct computation.
\end{proof}

\bibliographystyle{amsalpha}

\bibliography{ITG_S}

\end{document}